\documentclass[11pt]{article}
\usepackage{booktabs}
\usepackage{caption}
\usepackage{mathrsfs}
\usepackage{amsmath}
\usepackage{amsfonts,amsthm,amssymb,mathrsfs,bbding}
\usepackage{float}
\usepackage{graphics,epstopdf}
\usepackage{graphics,multicol}
\usepackage{graphicx}
\usepackage{color}
\usepackage{enumerate}
\usepackage{tikz}
\usepackage{caption}
\usepackage{rotating}
\usepackage{lscape}
\usepackage{longtable}
\allowdisplaybreaks[4]
\usepackage{tabularx}
\usepackage[colorlinks=true,anchorcolor=blue,filecolor=blue,linkcolor=blue,urlcolor=blue,citecolor=blue]{hyperref}
\usepackage{extarrows}
\usepackage{cite}
\usepackage{latexsym,bm}
\usepackage{mathtools}
\evensidemargin=\oddsidemargin\topmargin=-1.5cm

\newtheorem{thm}{Theorem}[section]

\newtheorem{lem}[thm]{Lemma}
\newtheorem{cor::3.1}[thm]{Corollary}

\theoremstyle{definition}
\renewcommand\proofname{\bf Proof}
\addtocounter{section}{0}
\begin{document}

\title{\bf The spanning $k$-trees, perfect matchings and spectral radius  of graphs}
\author{Dandan Fan$^a$, Sergey Goryainov$^b$, Xueyi Huang$^a$\footnote{Corresponding author.}\setcounter{footnote}{-1}\footnote{\emph{Email address:} ddfan0526@163.com (D. Fan), sergey.goryainov3@gmail.com(S. Goryainov), huangxymath@163.com (X. Huang), huiqiulin@126.com (H. Lin).}, Huiqiu Lin$^a$\\[2mm]
\small\it $^a$School of Mathematics, East China University of Science and Technology, \\
\small\it   Shanghai 200237, P.R. China\\[1mm]
\small\it  $^b$Department of Mathematics, Chelyabinsk State University, \\
\small\it   Chelyabinsk 454021, Russia}
\date{}
\maketitle
{\flushleft\large\bf Abstract}
A $k$-tree is a spanning tree in which every vertex has degree at most $k$.
In this paper, we provide a sufficient condition for the existence of a $k$-tree in a connected graph with fixed order in terms of the adjacency spectral radius and the signless Laplacian spectral radius, respectively. Also, we give a similar condition for the existence of a perfect matching in a balanced bipartite graph with fixed order and  minimum degree.

\begin{flushleft}
\end{flushleft}

\section{Introduction}
Let $G$ be a  graph with vertex set $V(G)$ and edge set $E(G)$. For any $v\in V(G)$, let $N_{G}(v)$ ($N(v)$ for short) and $d_G(v)$ ($d(v)$ for short) denote the neighborhood and degree of $v$, respectively. The \textit{adjacency matrix} of $G$ is defined as the matrix $A(G)=(a_{u,v})_{u,v\in V(G)}$ with $a_{u,v}=1$ if $u$ and $v$ are adjacent, and $a_{u,v}=0$ otherwise. The \textit{signless Laplacian matrix} of $G$ is defined as $Q(G)=D(G)+A(G)$, where $D(G)$ is the diagonal matrix of vertex degrees of $G$. The largest eigenvalues of $A(G)$ and $Q(G)$, denoted by $\rho(G)$ and $q(G)$, are called the  \textit{ajdacency spectral radius} and  \textit{signless Laplacian spectral radius} of $G$, respectively.

A basic result in graph theory asserts that every connected graph has a spanning tree. In the past several decades, the existence of a spanning tree with a given parameter in a connected graph has been widely studied by researchers (see, for example \cite{G.Ding,L.Gargano-1,L.Gargano-2,L.Gargano-3,L.Neumann,ci3}). A $k$-tree is a spanning tree with every vertex of degree at most $k$. It is natural to ask  which graphs contain a $k$-tree. In 1989, Win \cite {S.Win} (see also \cite{M.Ellingham} for a short proof) proved the following result, which gives a sufficient condition for the existence of a $k$-tree in a connected graph.
\begin{lem}(Win \cite{S.Win})\label{lem::1.1}
Let $G$ be a connected graph. If $k\geq 2$ and for any subset $S$ of $V(G)$
$$c(G-S)\leq(k-2)|S|+2,$$
then $G$ has a  $k$-tree, where $c(G-S)$ is the number of components in $G-S$.
\end{lem}

As usual, let $K_{n}$ denote the complete graph on $n$ vertices, and $K_{m,n}$ the complete bipartite graph with parts of size $m$ and $n$. Given two graphs $G_1$ and $G_2$, the \textit{join}  $G_{1} \nabla G_{2}$ is the graph obtained from  $G_{1}\cup G_{2}$ by adding all edges between $G_{1}$ and $G_{2}$.

Inspired by the work of Win \cite {S.Win}, it is interesting to find a spectral condition for the existence of a $k$-tree in a connected graph. Since a $2$-tree is just a  Hamilton path, the  following two theorems provide such conditions for $k=2$ in terms of the adjacency spectral radius and the signless Laplacian spectral radius, respectively.

\begin{thm}(Ning and Ge \cite{B.Ning})\label{thm::1.1}
Let $G$ be a connected graph of order $n\geq 4$. If $\rho(G)> n-3$, then $G$ contains a Hamilton path unless $G \in \{K_{1} \nabla (K_{n-3} \cup 2K_{1}),K_{2} \nabla 4 K_{1}, K_{1} \nabla  (K_{1,3} \cup K_{1})\}$.
\end{thm}

\begin{thm}(Liu, Shiu and Xue \cite{R.Liu})\label{thm::1.2}
Let $G$ be a connected graph of order $n\geq 4$. If $q(G)\geq 2n-5$, then $G$ contains a Hamilton path unless $G\in \{K_{2}\nabla 4 K_{1}, K_{1} \nabla  (K_{2}\cup 2K_{1}), K_{1,3}, K_{1,4}\}$.
\end{thm}

In this paper, for any $k\geq 3$,  we  give a spectral radius condition for the existence of a $k$-tree in a connected graph.
\begin{thm}\label{thm::1.3}
Let $k\geq 3$, and let  $G$ be a connected graph of order $n\geq 2k+16$.
\begin{enumerate}[(i)]
\item  If $\rho(G)\geq \rho(K_{1} \nabla (K_{n-k-1}\cup k K_{1}))$,
then $G$ contains a $k$-tree unless $G\cong K_{1} \nabla (K_{n-k-1}\cup kK_{1})$.
\item  If $q(G)\geq q(K_{1} \nabla (K_{n-k-1} \cup k K_{1}))$,
then $G$ contains a $k$-tree unless $G\cong K_{1} \nabla (K_{n-k-1}\cup kK_{1})$.
\end{enumerate}
\end{thm}

A \textit{matching} in a graph is a set of pairwise nonadjacent edges, and a \textit{perfect matching} is a matching that covers all vertices of the graph. The  relationship between the eigenvalues and the matching number of a graph was first studied by  Brouwer and Haemers \cite{A.Brouwer}. They described several sufficient conditions, in terms of the eigenvalues of the adjacency and Laplacian matrices, for a graph to contain a perfect matching. Up to now, much attention has been paid on this topic, and we refer the reader to \cite{A.Chang,S.Li,ci,ci1,ci2,S.O}. In addition,  the relationship between the eigenvalues and the fractional matching number (see \cite{SU} for the definition)  of a graph was also  investigated by several researchers in recent years \cite{O2,PLZ,XZS}.

The following fundamental theorem of Hall provided a sufficient and necessary condition for the existence of a perfect matching in a bipartite graph.

\begin{lem}(Hall's condition, \cite{Hall})\label{lem::1.2}
A bipartite graph $G=(X,Y)$ has a perfect matching if and only if $|X|=|Y|$ and $|N(S)|\geq|S|$ for all
$S\subseteq X$.
\end{lem}

A bipartite graph $G = (X, Y)$ is called \emph{balanced} if $|X| = |Y|$. Given two bipartite graphs $G_1=(X_{1}, Y_{1})$ and $G_2=(X_{2}, Y_{2})$, let $G_{1}\nabla_{1} G_2$ denote the graph obtained from $G_1 \cup G_2$ by adding all possible edges between $X_{2}$ and $Y_{1}$.

Very recently, O \cite{S.O} provided a lower bound for the adjacency spectral radius which guarantees the existence of a perfect matching in a graph. Motivated by Hall's condition and  the work of O, in this paper, we give a spectral radius condition for the existence of a perfect matching in a balanced bipartite graph  with minimum degree $\delta$.

\begin{thm}\label{thm::1.4}
Let $G$ be a balanced bipartite graph of order $2n$  with minimum degree $\delta\geq 1$.
\begin{enumerate}[(i)]
\item  If $\rho(G)\geq\rho(K_{\delta+1,\delta}\nabla_{1} K_{n-\delta-1,n-\delta})$,
then $G$ contains a perfect matching unless $G\cong K_{\delta+1,\delta}\nabla_{1} K_{n-\delta-1,n-\delta}$.
\item  If $q(G)\geq q(K_{\delta+1,\delta}\nabla_{1} K_{n-\delta-1,n-\delta})$,
then $G$ contains a perfect matching unless $G\cong K_{\delta+1,\delta}\nabla_{1} K_{n-\delta-1,n-\delta}$.
\end{enumerate}
\end{thm}

By a simple calculation,
$$\rho(K_{\delta+1,\delta}\nabla_{1} K_{n-\delta-1,n-\delta})=\frac{\sqrt{2n^2 - 2(\delta + 1)n +  2\delta(\delta+1) + 2f(n,\delta)}}{2},$$
where $f(n,\delta)=\sqrt{n^4 - 2(\delta + 1)n^3 + (1 - \delta^2)n^2 + 2\delta(3\delta^2 + 4\delta + 1)n - 3\delta^2(\delta+1)^2}$
and
$$q(K_{\delta+1,\delta}\nabla_{1} K_{n-\delta-1,n-\delta})=\frac{2n-1+\sqrt{4n^2 - ( 8\delta + 4)n + 8\delta^2 + 8\delta + 1}}{2}.$$
 Considering that the expression of $\rho(K_{\delta+1,\delta}\nabla_{1} K_{n-\delta-1,n-\delta})$  is complicated and the fact that $\rho(K_{\delta+1,\delta}\nabla_{1} K_{n-\delta-1,n-\delta})\geq \rho(K_{n-\delta-1,n}\cup (\delta+1)K_1)=\sqrt{n(n-\delta-1)}$,
  we provide the following result which improves Theorem \ref{thm::1.4} (i)  for sufficiently large $n$.

\begin{thm}\label{thm::1.5}
Let $G$ be a balanced bipartite graph of order $2n$ with minimum degree $\delta\geq 1$. If  $n\geq\frac{1}{2}\delta^{3}+\frac{1}{2}\delta^{2}+\delta+4$ and $\rho(G)\geq\sqrt{n(n-\delta-1)}$, then $G$ has a perfect matching unless $G\cong K_{\delta+1,\delta}\nabla_{1} K_{n-\delta-1,n-\delta}$.
\end{thm}

\section{Preliminaries}
In order to prove the main results for $\rho(G)$ and $q(G)$ simultaneously, we introduce the matrix $A_a(G)=aD(G)+A(G)$ and denote by $\rho_a(G)$ the largest eigenvalue of $A_a(G)$, where $a\geq 0$. Clearly, $A_0(G)=A(G)$ (resp. $A_1(G)=Q(G)$) and $\rho_0(G)=\rho(G)$ (resp. $\rho_1(G)=q(G)$).

The following lemma was proved for $A(G)$ and $Q(G)$ in \cite{Y.Hong-2,H.Liu}, and can be easily extended to the general matrix $A_a(G)$ for $a\geq 0$.

\begin{lem}\label{lem::2.1}
 Let $G$ be a connected graph, and let $u,v$ be two vertices of $G$.  Suppose that $v_{1},v_{2},\ldots,v_{s}\in N_{G}(v)\backslash N_{G}(u)$ with $s\geq 1$, and $G^*$ is the graph obtained from $G$ by deleting the edges $vv_{i}$ and adding the edges $uv_{i}$ for $1\leq i\leq s$. Let $x$ be the Perron vector of $A_a(G)$ where $a\geq 0$. If $x_{u}\geq x_{v}$, then $\rho_a(G)<\rho_a(G^*)$.
\end{lem}

Also, by the well-known Perron-Frobenius theorem (cf. \cite[Section 8.8]{C.Godsil}), we can easily deduce the following result.

\begin{lem}\label{lem::2.2}
Let $a\geq 0$. If $H$ is a spanning subgraph of a connected graph $G$, then
$$\rho_{a}(H)\leq\rho_{a}(G),$$
with equality if and only if $H\cong G$.
\end{lem}

\begin{lem}(Hong \cite{Y.Hong})\label{lem::2.3}
Let $G$ be a graph with $n$ vertices and $m$ edges. Then
                  $$\rho_0(G)=\rho(G)\leq\sqrt{2m-n+1}.$$
\end{lem}

\begin{lem}(Das \cite{K.Das})\label{lem::2.4}
Let $G$ be a graph with $n$ vertices and $m$ edges. Then
                      $$\rho_1(G)=q(G)\leq\frac{2m}{n-1}+n-2.$$
\end{lem}

Let $M$ be a real $n\times n$ matrix, and let $X=\{1,2,\ldots,n\}$. Given a partition $\Pi=\{X_1,X_2,\ldots,X_k\}$ with $X=X_{1}\cup X_{2}\cup \cdots \cup X_{k}$, the matrix $M$ can be partitioned as
$$
M=\left(\begin{array}{ccccccc}
M_{1,1}&M_{1,2}&\cdots &M_{1,k}\\
M_{2,1}&M_{2,2}&\cdots &M_{2,k}\\
\vdots& \vdots& \ddots& \vdots\\
M_{k,1}&M_{k,2}&\cdots &M_{k,k}\\
\end{array}\right).
$$
The \textit{quotient matrix} of $M$ with respect to $\Pi$ is defined as the $k\times k$ matrix $B_\Pi=(b_{i,j})_{i,j=1}^k$ where $b_{i,j}$ is the  average value of all row sums of $M_{i,j}$.
The partition $\Pi$ is called \textit{equitable} if each block $M_{i,j}$ of $M$ has constant row sum $b_{i,j}$.
Also, we say that the quotient matrix $B_\Pi$ is \textit{equitable} if $\Pi$ is an equitable partition of $M$.

\begin{lem}(Brouwer and Haemers \cite[p. 30]{BH}; Godsil and Royle\cite[pp. 196--198]{C.Godsil})\label{lem::2.5}
Let $M$ be a real symmetric matrix, and let $\lambda_{1}(M)$ be the largest eigenvalue of $M$. If $B_\Pi$ is an equitable quotient matrix of $M$, then the eigenvalues of  $B_\Pi$ are also eigenvalues of $M$. Furthermore, if $M$ is nonnegative and irreducible, then $\lambda_{1}(M) = \lambda_{1}(B_\Pi).$
\end{lem}

\section{Proof of Theorem \ref{thm::1.3}}
In this section,  we give the proof of Theorem \ref{thm::1.3}. Before doing this, we need the following lemma.

\begin{lem}\label{lem::3.1}
Let $a\geq 0$ and $n=\sum_{i=1}^t n_i+s$. If $n_{1}\geq n_{2}\geq \cdots\geq n_{t}\geq 1$ and $n_{1}<n-s-t+1$, then
$$\rho_{a}(K_{s} \nabla (K_{n_{1}}\cup K_{n_{2}}\cup \cdots \cup K_{n_{t}}))<\rho_{a}(K_{s} \nabla (K_{n-s-t+1}\cup (t-1)K_{1})).$$
\end{lem}

\begin{proof}
Let $G=K_{s} \nabla (K_{n_{1}}\cup K_{n_{2}}\cup \cdots \cup K_{n_{t}})$, and let $x$ denote the Perron vector of $A_{a}(G)$. By symmetry, we can suppose that $x(v)=x_i$ for all $v\in V(K_{n_{i}})$, where $1\leq i\leq t$, and $x(u)=y_1$ for all $u\in V(K_{s})$.
Then, from $A_{a}(G)x=\rho_{a}(G) x$, we have
$$(\rho_{a}(G)-((a+1)(n_{1}-1)+as)) x_{1}=sy_{1}>0,$$
which gives that $\rho_{a}(G)> (a+1)(n_{1}-1)+as$. Again, for $2\leq j\leq t$, we see that
$$(\rho_{a}(G)-((a+1)(n_{j}-1)+as))(x_{1}-x_{j})=(a+1)(n_{1}-n_{j})x_{1}\geq 0.$$
Since $\rho_{a}(G)>(a+1)(n_{1}-1)+as\geq (a+1)(n_{j}-1)+as$, we must have $x_{1}\geq x_{j}$ for $2\leq j\leq t$. Let $G'=K_{s} \nabla (K_{n-s-t+1}\cup (t-1)K_{1})$. Then
\begin{equation*}
\begin{aligned}
                         \rho_{a}(G')-\rho_{a}(G)&\geq x^{T}(A_{a}(G')-A_{a}(G))x\\
                                 &=2\left(\sum_{j=2}^{t}(n_{j}-1)x_{j}(n_{1}x_{1}-x_{j})+
                                    \sum_{i=2}^{t-1}\sum_{j=i+1}^{t}(n_{i}-1)(n_{j}-1)x_{i}x_{j}\right)\\
                                 &+a\sum_{j=2}^{t}(n_{j}-1)(n_{1}x^{2}_{1}+(n_{1}-2)x^{2}_{j})+
                                 a\sum_{i=2}^{t-1}\sum_{j=i+1}^{t}(n_{i}-1)(n_{j}-1)(x^{2}_{i}+x^{2}_{j})\\
                                  &>0,
\end{aligned}
\end{equation*}
and the result follows.
\end{proof}

Now, we are in a position to give the  proof of Theorem \ref{thm::1.3}.
\renewcommand\proofname{\bf Proof of Theorem \ref{thm::1.3}}
\begin{proof}
Assume that $G$ is a connected graph of order $n$ containing no $k$-trees for some $k\geq 3$ and $n \geq 2k+16$. By Lemma \ref{lem::1.1}, there exists some  nonempty subset $S$ of $V(G)$ such that $c(G-S)\geq(k-2)|S|+3$. Let $|S|=s$  and $t=(k-2)s+3$. Then  $G$ is a spanning subgraph of $G_s^1=K_{s} \nabla (K_{n_1}\cup K_{n_2}\cup \cdots \cup K_{n_t})$ for some positive integers $n_1\geq n_2\geq \cdots\geq n_t$ with  $\sum_{i=1}^tn_i=n-s$.
Let $a\in\{0,1\}$. By Lemma \ref{lem::2.2}, we obtain
\begin{equation}\label{equ::1}
\rho_{a}(G)\leq\rho_{a}(G_s^1),
\end{equation}
where equality holds if and only if $G\cong G_s^1$.

Take $G_s^2=K_{s} \nabla (K_{n-s-t+1}\cup (t-1)K_{1})$. Note that $t=(k-2)s+3$. By Lemma \ref{lem::3.1},
\begin{equation}\label{equ::2}
\rho_{a}(G_s^1)\leq \rho_{a} (G_{s}^2),
\end{equation}
with equality holding if and only if $(n_1,\ldots,n_t)=(n-s-t+1,1,\ldots,1)$.

If $s=1$, then $G_{s}^2\cong K_{1} \nabla (K_{n-k-1}\cup kK_{1})$.
Combining this with (\ref{equ::1}), we may conclude that
$$\rho_a(G)\leq\rho_a(K_{1} \nabla (K_{n-k-1}\cup kK_{1})),$$
where equality holds if and only if $G\cong K_{1} \nabla (K_{n-k-1}\cup kK_{1})$.
For $s\geq 2$, we consider the two cases $a=0$ and $a=1$.

{\flushleft\bf Case 1.} $a=0.$

By Lemma \ref{lem::2.3},
\begin{equation}\label{equ::3}
\begin{aligned}
\rho_{0}(G_{s}^2)&\leq\sqrt{2m(G_{s}^2)-n+1}\\
&=\sqrt{[(n-(k-2)s-2)(n-(k-2)s-3)+2((k-2)s+2)s]-n+1}\\
&=\sqrt{(k^{2}-2k)s^{2}-(2kn-5k-4n+6)s+n^{2}-6n+7}.
\end{aligned}
\end{equation}
Let $f(s)=(k^{2}-2k)s^{2}-(2kn-5k-4n+6)s+n^{2}-6n+7$. Since $n\geq s+t=(k-1)s+3$, $2 \leq s\leq(n-3)/(k-1)$. By a simple calculation,
\begin{equation*}
\begin{aligned}
&~~~~f(2)-f\Big(\frac{n-3}{k-1}\Big)\\
&=\frac{(n-2k-1)((k-2)^{2}n-2k^3 +4k^2 +k-6)}{(k-1)^{2}}\\
&\geq\frac{15((k-2)^{2}(2k+16)-2k^3 +4k^2 +k-6)}{(k-1)^{2}}\\
&=\frac{15(12k^{2}-55k+58)}{(k-1)^{2}}\\
&>0,
\end{aligned}
\end{equation*}
where the inequalities follow from the fact that $n\geq 2k+16$ and  $k\geq 3$.
This implies that, for $2\leq s\leq(n-3)/(k-1)$, the maximum value of $f(s)$ is attained at $s=2$, and so from (\ref{equ::3}), we deduce that
\begin{equation}\label{equ::4}
\begin{aligned}
\rho_{0}(G_{s}^2)&\leq \sqrt{f(2)}\\
&=\sqrt{(n-k-1)^{2}-((2k - 4)n-3k^{2}+6)}\\
&\leq \sqrt{(n-k-1)^{2}-((2k - 4)(2k+16) - 3k^2 +6)}\\
&=\sqrt{(n-k-1)^{2}-(k^2+24k-58)}\\
                                 &< n-k-1,\\
\end{aligned}
\end{equation}
where the penultimate inequality follows from the fact that $n\geq 2k+16$.
 Since $K_{1} \nabla (K_{n-k-1}\cup kK_{1})$ contains $K_{n-k}\cup k K_{1}$ as a proper spanning  subgraph, by Lemma \ref{lem::2.2}, $\rho_0(K_{1} \nabla (K_{n-k-1} \cup kK_{1}))>\rho_0(K_{n-k}\cup kK_{1})=n-k-1$. Combining this with (\ref{equ::1}), (\ref{equ::2}) and (\ref{equ::4}), we conclude that $\rho_0(G)\leq \rho_0(G_s^1)\leq \rho_0(G_s^2)<\rho_0(K_{1} \nabla (K_{n-k-1}\cup kK_{1}))$.

{\flushleft\bf Case 2.} $a=1.$

According to Lemma \ref{lem::2.4},
\begin{equation}\label{equ::5}
\begin{aligned}
                         \rho_{1}(G_{s}^{2})&\leq\frac{2m(G_{s}^{2})}{n-1}+n-2\\
                                 &=\frac{(n-(k-2)s-2)(n-(k-2)s-3)+2((k-2)s+2)s}{n-1}+n-2\\
                                 &=\frac{(k^{2}-2k)s^{2}+(4n-2kn+5k-6)s+2(n-2)^{2}}{n-1}.
\end{aligned}
\end{equation}
Let $g(s)=(k^{2}-2k)s^{2}-(2kn-5k-4n+6)s+2(n-2)^{2}$. Recall that $2\leq s\leq(n-3)/(k-1)$. A simple calculation yields that
\begin{equation*}
\begin{aligned}
                        &~~~~g(2)-g\Big(\frac{n-3}{k-1}\Big)\\
                        &=\frac{(n-2k-1)((k-2)^{2}n-2k^3+4k^2+k-6)}{(k-1)^{2}}\\
                         &\geq\frac{15(12k^2-55k+58)}{(k-1)^{2}}\\
                                 &>0,
\end{aligned}
\end{equation*}
where the inequalities follow from the fact that $n\geq 2k+16$ and  $k\geq 3$.
This implies that, for $2\leq s\leq(n-3)/(k-1)$, the maximum value of $g(s)$ is attained at $s=2$, and so from (\ref{equ::5}), we obtain
\begin{equation}\label{equ::6}
\begin{aligned}
                        \rho_{1}(G_{s}^{2})&\leq \frac{g(2)}{n-1}\\
                         &=2(n-k-1)-\frac{(2k-4)n-4k^{2}
                                 +6}{n-1}\\
                                 &\leq 2(n-k-1)-\frac{(2k-4)(2k+16)-4k^{2}+6}{n-1}\\
                                 &\leq 2(n-k-1)-\frac{24k-58}{n-1}\\
                                  &< 2(n-k-1),
\end{aligned}
\end{equation}
where the penultimate inequality follows from the fact that $n\geq 2k+16$.
 Since $K_{1} \nabla (K_{n-k-1}\cup kK_{1})$ contains $K_{n-k}\cup kK_{1}$ as a proper spanning  subgraph, by Lemma \ref{lem::2.2}, $\rho_1(K_{1} \nabla (K_{n-k-1}\cup kK_{1}))>\rho_1(K_{n-k}\cup kK_{1})=2(n-k-1)$. Combining this with (\ref{equ::1}), (\ref{equ::2}) and (\ref{equ::6}), we have $\rho_1(G)\leq \rho_1(G_s^1)\leq\rho_1(G_s^2)<\rho_1(K_{1} \nabla (K_{n-k-1}\cup kK_{1}))$.

Concluding the above results, we obtain
$$\rho_a(G)\leq\rho_a(K_{1} \nabla (K_{n-k-1}\cup kK_{1}))$$
for $a\in\{0,1\}$, where equality holds  if and only if $G\cong K_{1} \nabla (K_{n-k-1}\cup kK_{1})$. Let  $u$ denote the unique vertex of degree $n-1$ in $K_{1} \nabla (K_{n-k-1}\cup kK_{1})$. Since $u$ is adjacent to  $k$ pendant vertices, the degree of  $u$ must be at least  $k+1$ in each spanning tree of $K_{1} \nabla (K_{n-k-1}\cup kK_{1})$. This suggests that $K_{1} \nabla (K_{n-k-1}\cup kK_{1})$ contains no $k$-trees, and so the result follows.\end{proof}

\section{Proof of Theorem \ref{thm::1.4} and Theorem \ref{thm::1.5}}

In order to prove Theorems 1.4 and 1.5, we require the following lemma.

\begin{lem} \label{lem::4.1}
Let $a=0$ or $1$. For $1\leq s<n/2$,
$$\rho_{a}(K_{s+1,s}\nabla_{1} K_{n-s-1,n-s})<\rho_{a}(K_{s,s-1}\nabla_{1} K_{n-s,n-s+1}).$$
\end{lem}

\renewcommand\proofname{\bf Proof}
\begin{proof}
Observe that $A_{a}(K_{s+1,s}\nabla_{1} K_{n-s-1,n-s})$ has the equitable quotient matrix
$$
B_\Pi^s=\begin{bmatrix}
as &0 & s& 0\\
0&an &s &n-s\\
  s+1&n-s-1& an& 0\\
   0 &n-s-1 &0 &a(n-s-1)
\end{bmatrix}.
$$
By a simple calculation, the characteristic polynomial of $B_\Pi^s$ is
\begin{equation*}
\begin{aligned}
\varphi(B_{\Pi}^s,x)&=x^4-a(3n-1)x^3-[(a^2 + 1)s^2 - (a^2 + 1)(n - 1)s + (1-3a^2)n^2 - (1 - 2a^2)n]x^{2}\\
&~~~+an[2a^2s^2 - 2a^2(n-1)s + (1 - a^2)(n^2-n)]x\\
&~~~+s(n-s-1)(a^{2}-1)(a^2n^2-ns+s^2-n+s).
\end{aligned}
\end{equation*}
Also, note that $A_{a}(K_{s,s-1}\nabla_{1} K_{n-s,n-s+1})$ has the equitable quotient matrix $B_{\Pi}^{s-1}$, which is obtained by replacing $s$ with $s-1$ in $B_\Pi^s$. Then
\begin{equation}\label{equ::7}
\varphi(B_{\Pi}^s,x)-\varphi(B_{\Pi}^{s-1},x)=(n-2s)[(a^2+1)x^2-2a^{3}nx-(1-a^2)(2s^2 -2ns + a^2n^2)].
\end{equation}
Recall that $s<n/2$. For $a=0$, from (\ref{equ::7}), we obtain
$$\varphi(B_{\Pi}^s,x)-\varphi(B_{\Pi}^{s-1},x)=(n-2s)(x^2+2ns-2s^2)>0,$$
which leads to
$\lambda_{1}(B_{\Pi}^s)<\lambda_{1}(B_{\Pi}^{s-1})$. By Lemma \ref{lem::2.5},
$$\rho_{0}(K_{s+1,s}\nabla_{1} K_{n-s-1,n-s})<\rho_{0}(K_{s,s-1}\nabla_{1} K_{n-s,n-s+1}).$$
For $a=1$, since $K_{s,s-1}\nabla_{1}K_{n-s,n-s+1}$ contains  $K_{1,n-1}$ as a proper subgraph, we have  $\rho_{1}(K_{s,s-1}\nabla_{1} K_{n-s,n-s+1})>\rho_{1}(K_{1,n-1})=n$. Thus, by (\ref{equ::7}),
$$\varphi(B_{\Pi}^s,x)-\varphi(B_{\Pi}^{s-1},x)=2(n-2s)x(x-n)>0$$
for $x\geq \rho_{1}(K_{s,s-1}\nabla_{1} K_{n-s,n-s+1})$, which implies that
$$\rho_{1}(K_{s+1,s}\nabla_{1} K_{n-s-1,n-s})<\rho_{1}(K_{s,s-1}\nabla_{1} K_{n-s,n-s+1}).$$\end{proof}

By using Lemma \ref{lem::4.1}, we now give a short proof of Theorem \ref{thm::1.4}.

\renewcommand\proofname{\bf Proof of Theorem \ref{thm::1.4}}
\begin{proof}

Assume that $G$ is a balanced bipartite graph of order $2n$ with minimum degree $\delta$ having no perfecting matchings. By Lemma \ref{lem::1.2}, there exists some  nonempty $X_{1}\subseteq X$ such that  $|N(X_{1})|<|X_1|$. It follows that $G$ is a spanning subgraph of $K_{s+1,s}\nabla_{1}K_{n-s-1,n-s}$ for some $s$ with $\delta \leq s< n/2$. Then, from Lemmas \ref{lem::2.2} and \ref{lem::4.1},
\begin{equation*}
\begin{aligned}
\rho_{a}(G)&\leq\rho_{a}(K_{s+1,s}\nabla_{1} K_{n-s-1,n-s})\\
&\leq\rho_{a}(K_{\delta+1,\delta}\nabla_{1} K_{n-\delta-1,n-\delta}),
\end{aligned}
\end{equation*}
 where the first equality holds if and only if $G\cong K_{s+1,s}\nabla_{1}K_{n-s-1,n-s}$, and the second  holds if and only if $s=\delta$. Also, $K_{\delta+1,\delta}\nabla_{1} K_{n-\delta-1,n-\delta}$ contains no perfect matchings, and so the result follows.
\end{proof}

To prove Theorem \ref{thm::1.5}, we need the following lemmas.
\begin{lem}\label{lem::4.2}
Let $\delta \geq 1$ and $n\geq\frac{1}{2}\delta^{3}+\frac{1}{2}\delta^{2}+\delta+4$. If $G$ is a spanning subgraph of  $K_{\delta+1,\delta}\nabla_{1} K_{n-\delta-1,n-\delta}$ with minimum degree $\delta$, then $\rho(G)<\sqrt{n(n-\delta-1)}$ unless $G=K_{\delta+1,\delta}\nabla_{1} K_{n-\delta-1,n-\delta}$.
\end{lem}

\renewcommand\proofname{\bf Proof}
\begin{proof}
By Lemma \ref{lem::2.2}, we note that  $\rho(K_{\delta+1,\delta}\nabla_{1} K_{n-\delta-1,n-\delta})>\rho(K_{n-\delta-1,n}\cup (\delta+1)K_1)=\sqrt{n(n-\delta-1)}$. Suppose that $G$ is a proper spanning subgraph of $K_{\delta+1,\delta}\nabla_{1} K_{n-\delta-1,n-\delta}$ with minimum degree $\delta$. Then $G$ is exactly a spanning subgraph of the graph $H$ obtained by removing some edge $uv$ from $K_{\delta+1,\delta}\nabla_{1} K_{n-\delta-1,n-\delta}$, and by Lemma \ref{lem::2.2}, $\rho(G)\leq\rho(H)$. Thus it suffices to show that $\rho(H)<\sqrt{n(n-\delta-1)}$.

In the graph $K_{\delta+1,\delta}\nabla_{1} K_{n-\delta-1,n-\delta}$, let $X_{1}$ denote the set of the $\delta+1$ vertices with  degree $\delta$, $Y_{1}$ the set of vertices adjacent to $X_{1}$, $X_{2}$ the set of vertices of degree $n$, and $Y_{2}$ is the set of the remaining $n-\delta$ vertices.  Since the minimum degree of $H$ is  $\delta$, we only need to consider the cases that $u\in Y_1$ and $v\in X_2$ or $u\in X_2$ and $v\in Y_2$. For the former case, we denote  $H$ by $H_1$, and for the later, we denote $H$ by $H_2$.

Let $x$ be the Perron vector of $A(H_{1})$ with respect to $\rho(H_{1})$. We choose a vertex $w\in Y_{2}$, and let $H'$ be the graph obtained from $H_{1}$ by removing the edge $vw$ and adding the edge $uv$. Then we see that  $H'\cong H_{2}$. If $x(u)\geq x(w)$, then by Lemma \ref{lem::2.1},  $\rho(H_2)=\rho(H')>\rho(H_{1})$.
If $x(w)> x(u)$, we construct a new vector $x'$ which is obtained from $x$ by swapping the entries $x(u)$ and $x(w)$. As $x'$ is also a positive unit vector, by the Reyleigh quotient,
\begin{equation*}
\begin{aligned}
                        \rho(H_2)-\rho(H_{1}) &=\rho(H')-\rho(H_{1}) \\
                        &\geq x'^{T}A(H')x'-x^{T}A(H_{1})x\\
                                 &=2(x(w)-x(u))\sum_{v_1\in X_{1}}x(v_1)\\
                                 & >0.
\end{aligned}
\end{equation*}
Therefore $\rho(H_2)>\rho(H_1)$ and  it remains to prove that $\rho(H_2)<\sqrt{n(n-\delta-1)}$.

Assume to the contrary that $\rho(H_2)\geq\sqrt{n(n-\delta-1)}$. Let $y$ be the Perron vector of $A(H_{2})$ with respect to $\rho(H_{2})$, and let $\rho=\rho(H_{2})$. By symmetry,  $y$ takes the same values  (say $e$, $b$, $c$ and $d$) on the vertices of $X_1$, $X_2\setminus\{u\}$, $Y_1$ and $Y_2\setminus\{v\}$, respectively. Let $y(u)=b_{1}$ and $y(v)=d_{1}$. Then, from $A(H_{2})y=\rho y$, we obtain
$$
\begin{aligned}
   \rho e&=\delta c, \\
    \rho b&=\delta c+(n-\delta-1)d+d_{1},  \\
    \rho c&=(\delta+1)e+(n-\delta-2)b+b_{1},\\
    \rho d&=(n-\delta-2)b+b_{1},\\
    \rho b_{1}&=\delta c+(n-\delta-1)d,~\mbox{and}\\
    \rho d_{1}&=(n-\delta-2)b.
  \end{aligned}
$$
Thus
\begin{equation}\label{equ::8}
\begin{aligned}
                   e&=\frac{\delta c}{\rho}, \\
                  b_{1}&=\frac{(n-\delta-1)(\rho^{2}-\delta(\delta+1))+\delta\rho^{2}}{\rho^{3}}c,~\mbox{and}\\
                   d_{1}&=\frac{(\rho^{2}-\delta(\delta+1))(\rho^{2}-(n-\delta-1))-\delta\rho^{2}}{\rho^{4}}c.
  \end{aligned}
\end{equation}
Let $y'$ be the vector obtained from $y$ by  restricting to the vertex set of the subgraph $K_{n-\delta-1,n}$ of $K_{\delta+1,\delta}\nabla_{1} K_{n-\delta-1,n-\delta}$. We have
\begin{equation}\label{equ::9}
\begin{aligned}
y'^{T}A(K_{n-\delta-1,n})y'&=y^{T}A(H_{2})y+2b_{1}d_{1}-2\delta (\delta+1) ec=\rho+2b_{1}d_{1}-2\delta(\delta+1)ec\\
&\leq \rho(K_{n-\delta-1,n})(1-(\delta+1)e^{2})=\sqrt{n(n-\delta-1)}(1-(\delta+1)e^{2}).
\end{aligned}
\end{equation}
Recall that $\rho\geq\sqrt{n(n-\delta-1)}$. Then, from (\ref{equ::8}) and (\ref{equ::9}), we can deduce that
\begin{equation*}
\begin{aligned}
                      \delta^{2}(\delta+1)&\geq \frac{2\rho b_{1}d_{1}}{c^{2}}\\
                                 &=2\frac{(n-\delta-1)(\rho^{2}-\delta(\delta+1))+\delta\rho^{2}}{\rho^{2}}\cdot\frac{(\rho^{2}-\delta(\delta+1))
                                 (\rho^{2}-(n-\delta-1))-\delta\rho^{2}}{\rho^{4}}\\
                                 &\geq2(n-\delta-1)\left[\frac{(\rho^{2}-\delta(\delta+1))^{2}(\rho^{2}-(n-\delta-1))}{\rho^{6}}-\frac{\delta(\rho^{2}-\delta(\delta+1))
                             }{\rho^{4}}\right]\\
                              &=2(n-\delta-1)\left[(1-\frac{\delta(\delta+1)}{\rho^{2}})^{2}(1-\frac{n-\delta-1}{\rho^{2}})-\frac{\delta}{\rho^{2}}
                              (1-\frac{\delta(\delta+1)}{\rho^{2}})\right]\\
                               &>2(n-\delta-1)\left(1-\frac{n-\delta-1}{\rho^{2}}-\frac{2\delta(\delta+1)}{\rho^{2}}-\frac{\delta}{\rho^{2}}\right)\\
                               &\geq2(n-\delta-1)\left(1-\frac{n+2\delta^{2}+2\delta-1}{n(n-\delta-1)}\right)\\
                               &=2(n-\delta-1)-2-\frac{4\delta^{2}+4\delta-2}{n}\\
                               &>2(n-\delta-1)-6\\
                               &\geq  \delta^{2}(\delta+1),
\end{aligned}
\end{equation*}
where the last two inequalities follow from the fact that $n\geq\frac{1}{2}\delta^{3}+\frac{1}{2}\delta^{2}+\delta+4$. This is a contradiction, and so the result follows.
\end{proof}

\begin{lem}(Nosal \cite{Nosal})\label{lem::4.3}
Let $G$ be a bipartite graph. Then
                  $$\rho(G)\leq \sqrt{e(G)}.$$
\end{lem}

\renewcommand\proofname{\bf Proof of Theorem \ref{thm::1.5}}
\begin{proof}
Assume that $G$ is a balanced bipartite graph of order $2n$ with minimum degree $\delta$ having no perfecting matchings, by Lemma \ref{lem::1.2}, there exists some nonempty $S\subseteq X$ such that $|N(S)|<|S|$. It is easy to verify that $G$ is a spanning subgraph of $K_{s+1,s}\nabla_{1}K_{n-s-1,n-s}$ for some $s$ with $\delta \leq s< n/2$. If $s=\delta$, then by Lemma \ref{lem::4.2} and Theorem \ref{thm::1.4}, we obtain the result immediately. Thus, we consider $s\geq\delta+1$ in the following. Since the minimum degree of $G$ is $\delta$, there exists a vertex $v$ with $d_{G}(v)=\delta$. This implies that $G$ is a proper spanning subgraph of  $K_{s+1,s}\nabla_{1} K_{n-s-1,n-s}$. In the graph $K_{s+1,s}\nabla_{1} K_{n-s-1,n-s}$, let $X_{1}$ denote the set of the $s+1$ vertices with degree $s$, $Y_{1}$ the set of vertices adjacent to $X_{1}$, $X_{2}$ the set of vertices of degree $n$, and $Y_{2}$ is the set of the remaining $n-s$ vertices. If $v\in X_1$, then
\begin{equation*}
\begin{aligned}
e(G)&\leq s(s+1)+(n-s-1)n-(s-\delta)\\
     &= n(n-\delta-1)-(-s^2+sn-(n+1)\delta)\\
     &\leq n(n-\delta-1)-(n-\delta^2-3\delta-1)~(\mbox{since $\delta+1\leq s< n/2$})\\
     &\leq n(n-\delta-1)-\left(\frac{\delta^3}{2}-\frac{\delta^2}{2}-2\delta+3\right) ~(\mbox{since $n\geq\frac{1}{2}\delta^{3}+\frac{1}{2}\delta^{2}+\delta+4$})\\
     &< n(n-\delta-1) ~(\mbox{since $\delta\geq 1$}).
\end{aligned}
\end{equation*}
Combining this with Lemma \ref{lem::4.3}, we have $\rho(G)<\sqrt{n(n-\delta-1)}$. For the remaining three cases that $v\in X_2$ or $v\in Y_i$ where $i=1,2$, we can also deduce $\rho(G)<\sqrt{n(n-\delta-1)}$ by using similar analysis as above.

This completes the proof.
\end{proof}

\section*{Acknowledgements}

The authors would like to thank the editor and the two anonymous referees for their valuable comments and helpful suggestions. H. Lin is supported by the National Natural Science Foundation of China (Grant Nos. 12011530064 and 11771141). S. Goryainov is supported by RFBR according to the research project 20-51-53023. X. Huang is supported by the National Natural Science Foundation of China (Grant No. 11901540).

\end{document}